\documentclass[11pt,a4paper]{article}
\usepackage{latexsym}
\usepackage{amsmath}
\usepackage{amsthm}
\usepackage{amssymb}
\usepackage[dvips]{graphicx}
\usepackage{times}
\usepackage{amsmath, amsthm, amscd, amsfonts, amssymb, graphicx, color}
\usepackage[bookmarksnumbered, colorlinks, plainpages]{hyperref}
\usepackage{wrapfig}
\usepackage{subfig}
\begin{document}

% Line spacing -----------------------------------------------------------
%%\setlength{\topmargin}{0.0cm} \setlength{\textheight}{21.4cm}
%\setlength{\textwidth}{15.6cm} \setlength{\oddsidemargin}{.4cm}
%\setlength{\evensidemargin}{.4cm} \setlength{\parskip}{1ex}
%\setlength{\parindent}{1cm}

  %%%%%%%%%%%%%%%%%%%%%%%%%%%%%%%%%%%%5

\flushbottom
%%%%%%%%%%%%%%%%%%%%%%%%%%%%%%%%%%%%%%%%%%%%%
\newcommand{\bt}{\beta}
\newcommand{\al}{\alpha}
\newcommand{\laa}{\lambda_\alpha}
\newcommand{\lab}{\lambda_\beta}
\newcommand{\no}{|\Omega|}
\newcommand{\nd}{|D|}
\newcommand{\la}{\lambda}
\newcommand{\ro}{\varrho}
\newcommand{\cd}{\chi_{D}}
\newcommand{\cdc}{\chi_{D^c}}
\newcommand{\be}{\begin{equation}}
\newcommand{\ee}{\end{equation}}
\newcommand{\Om}{\Omega}
\newcommand{\h}{H^1_0(\Omega)}
\newcommand{\lt}{L^2(\Omega)}
\newtheorem{thm}{Theorem}[section]
\newtheorem{cor}[thm]{Corollary}
\newtheorem{lem}[thm]{Lemma}
\newtheorem{prop}[thm]{Proposition}
\theoremstyle{definition}
\newtheorem{defn}{Definition}[section]
\theoremstyle{remark}
\newtheorem{rem}{Remark}[section]
\numberwithin{equation}{section}
\renewcommand{\theequation}{\thesection.\arabic{equation}}
\numberwithin{equation}{section}
%\setlength{\tclineskip}{1.05\baselineskip}

%%%%%%%%%%%%%%%%%%%%%%%%%%%%%%%%%%%%%%%%%%%%%%
\title{Optimal ground state energy of  two-phase conductors
}
\author{\bf Abbasali Mohammadi \\
\small  Department of Mathematics, College of Sciences,\\
\small Yasouj University, Yasouj, Iran, 75918-74934\\
\small mohammadi@yu.ac.ir\\
\bf Mohsen Yousefnezhad \\
\small  Department of Mathematical Sciences,\\
\small  Sharif University of Technology, Tehran, Iran\\
\small yousefnezhad@mehr.sharif.ir
}

\maketitle \hrule
\textbf{Abstract.} Consider the problem of distributing two conducting materials in a
ball with fixed proportion in order to minimize the first eigenvalue
of a Dirichlet operator. It was conjectured that the optimal
distribution consists of putting the material with the highest conductivity
in a ball around the center. In this paper, we show that the conjecture
is false  for all dimensions $n \geq 2$.\\\\
%%%%%%%%%%%%%%%%%%%%%%%%%%%%%%%%%%%%%%%%%%%%%%%%%%%%%%%%%
\footnote{{\it Ams subject classifications :} Primary,  49Q10, 35Q93; Secondary,  35P15, 33C10 }
\hrule
%%%%%%%%%%%%%%%%%%%%%%%%%%%%%%%%%%%%%%%%%%%%%%%%%%%%%%%%%
%\begin{section}{Introduction}

\section{Introduction}\label{intro}
 $\quad$  Let $\Omega$ be a bounded domain in $\mathbb{R}^n$
with a smooth boundary which is to be called the design region and
consider two conducting materials with conductivities  $0<\alpha< \beta$.  These materials
are distributed in $\Om$ such that the volume of the region $D$
occupied by the material with conductivity $\beta$ is a fixed number $A$ with $0<A<|\Om|$.
 Consider the
following two-phase eigenvalue problem
\begin{align} \label{mainpde}
-\mathrm{div} \left(  (\beta\chi_D+\alpha\chi_{D^c})  \nabla u  \right)=\lambda u \quad
\mathrm{in}\:\Omega,\qquad \quad u=0 \quad \mathrm{on}\:\partial
\Omega,
\end{align}
where $\beta\chi_D+\alpha\chi_{D^c}$ is the conductivity,
  $\la$ is the ground state energy or the
smallest positive eigenvalue and $u$ is the corresponding eigenfunction.

 We use notation $\la(D)$ to show the dependence of the eigenvalue on $D$, the
 region with the highest conductivity.
 To determine
the system's profile which gives  the minimum principal
eigenvalue, we should verify the following optimization problem
\begin{equation}\label{mainopt}
\underset{D\subset \Om,\:|D|= A} {\inf}\lambda(D),
\end{equation}
where $\la$ has the following variational formulation
\begin{equation}\label{lambdavar}
\lambda(D)=\underset{u\in H^1_0(\Om),\: \|u\|_{\lt}=1} {\min}\int_{\Om} ( \beta\chi_D+\alpha\chi_{D^c})  |\nabla u|^2 dx.
\end{equation}

 In general, this problem has no solution in any class of usual domains.
 Cox and Lipton have given in \cite{coxlipton} conditions for an optimal
microstructural design. However, when $\Om$ is a ball, the symmetry of the domain implies
that there exists a radially symmetric minimizer.
 Alvino \emph{et al} have obtained this result thanks to a comparison
result for Hamilton-Jacobi equations \cite{Alvino}. Conca \emph{et al}.
 have revived interest in this problem by giving a new simpler proof
 of the existence result only using rearrangement techniques
 \cite{Conca1}.

 %In general, the problems of optimal design may not admit solutions if
% one excludes the microstructural designs \cite{murat}.
% It is known that problem \eqref{mainopt}  has  no solution except in case of
%  including microstructural designs
% \cite{coxlipton}. However, this problem
% may still have a solution for specific $\Om$, for instance when
% $\Om$ is a ball. For one-dimensional case, Krein has shown in
% \cite{Kerien} that the unique minimizer is an interval. For higher dimensions, the existence of a radially symmetric optimal set has
% been established in \cite{Alvino} when $\Om$  is
% a ball centered at the origin with radius $\mathcal{R}$, $\mathcal{B}(0,\mathcal{R})$. Conca \emph{et al}.
% have revived interest in this problem by giving a new simpler proof
% of the existence result only using rearrangement techniques
% \cite{Conca1}.

  In eigenvalue optimization for elliptic partial differential equations, one of challenging mathematical problems after the problem of existence is an exact formula of the optimizer or optimal shape design. Most papers in this field answered this question just in case $\Om$ is a ball  \cite{chanillo, emamipro, emamielec-1, abbasali,abbasali2}.   This class of problems is difficult to solve due to the lack of the topology information of the optimal shape.
  For one-dimensional case, Krein has shown in
 \cite{Kerien} that the unique minimizer of \eqref{mainopt} is obtained by putting the material with the highest conductivity in an interval in the middle of the domain.
    Surprisingly, the exact distribution of the two materials which solves optimization problem \eqref{mainopt}  is still not known for higher dimensions.

   Let $\Om=\mathcal{B}(0,\mathcal{R})$ be a ball centered at the origin with radius $\mathcal{R}$ , the solution of the one-dimensional problem suggests for higher dimensions  that $\mathcal{B}(0,\mathcal{R}^*)$ is a  natural candidate to be the optimal domain.
 This conjecture has been supported by numerical evidence  in \cite{Conca2} using the shape derivative analysis of the first eigenvalue for the two-phase conduction problem. In addition, it
 has been shown in \cite{Ketab} employing the second order shape derivative
 calculus that $D=\mathcal{B}(0,\mathcal{R}^*)$ is a local strict minimum for the
 optimization problem \eqref{mainopt} when $A$ is small enough. In spite of the
 above evidence, it has been established in \cite{Conca3} that the conjecture is not
 true  in two- or three- dimensional spaces when $\al$ and
 $\bt$ are close to each other (low contrast regime) and $A$ is
 sufficiently large.  The theoretical base for the result is an asymptotic
  expansion of the eigenvalue with respect  to $\beta- \alpha$ as $\beta\rightarrow
  \alpha$, which allows one to approximate the optimization problem by
  a simple  minimization problem.

  In this paper, we investigate the conjecture for all dimensions $n\geq 2$.
 We prove that the conjecture is false not only for  two- or three- dimensional spaces, but also for all  dimensions  $n\geq 2$.  We have provided a different proof of the main result in \cite{Conca3}     and we will establish it  in a vastly simpler way.

%%%%%%%%%%%%%%%%%%%%%%%%%%%%%%%%%%%%%%%%%%%%%%%%%%%%%%%%%%%%%%%%%%%%%%%%%%%%%%%%%%%%%%%%%%%%%%%%%%%%%%%

\section{Preliminaries }\label{prel}

 $\quad$ In order to establish the main theorem, we need some preparation. Our proof is based upon the properties of Bessel functions. In this
section, we state some results from the theory of Bessel functions. The reader can refer to \cite{Bowman, Watson} for further
information about Bessel functions.

 Consider the standard form of  Bessel equation which is given by
 \begin{equation}\label{bessel eq}
 x^2y^{\prime\prime} +xy^{\prime}+(x^2-\nu^2)y=0,
  \end{equation}
 where $\nu$ is a nonnegative real number. The regular solution of \eqref{bessel eq}, called the Bessel function of the first kind  of order $\nu$, is given by
  \begin{equation}\label{besselseri}
J_\nu(x)=\sum_{k=0}^\infty \frac{(-1)^k x^{2k+\nu}}{2^{2k+\nu} \Gamma (\nu+k+1)},
  \end{equation}
  where $\Gamma$ is the gamma function. We shall use following recurrence relations between Bessel functions
   \begin{equation}\label{besselrec1}
J_{\nu-1}(x)+J_{\nu+1}(x)= \frac{2\nu}{x} J_\nu(x),
  \end{equation}
   \begin{equation}\label{besselrec2}
(x^{-\nu}J_{\nu}(x))^\prime=-x^{-\nu}J_{\nu+1}(x).
  \end{equation}

   Let $j_{\nu,m}$ be the $m$th positive zeros of the function $J_{\nu}(x)$, then it is well known that the zeros of $J_{\nu}(x)$ are simple with possible
  exception of $x=0$. In addition, we have the following lemma related to the roots of $J_{\nu}(x)$, \cite{Bowman, Watson}.
  \begin{lem}\label{besselroot}
   When $\nu\geq 0$, the positive roots of $J_{\nu}(x)$ and $J_{\nu+1}(x)$ interlace according to the inequalities
   \begin{equation*}
j_{\nu,m}< j_{{\nu+1},m}<j_{\nu,{m+1}}.
  \end{equation*}
     \end{lem}
      We will need the following technical assertion later.
      \begin{lem}\label{goodrel}
      If $\nu_1,\nu_2\geq0$, then
      \begin{equation*}
(\nu_2^2-\nu_1^2)\int_0^\tau \frac{J_{\nu_2}(s)J_{\nu_1}(s)}{s}\;ds= \tau (J^\prime_{\nu_2}(\tau)J_{\nu_1}(\tau)-J_{\nu_2}(\tau)J^{\prime}_{\nu_1}(\tau)).
  \end{equation*}
      \end{lem}
      \begin{proof}
      Functions $J_{\nu_2}$ and $J_{\nu_1}$ are solutions of Bessel equations
      \begin{equation*}
 x^2J_{\nu_2}^{\prime\prime} +xJ_{\nu_2}^{\prime}+(x^2-\nu_2^2)J_{\nu_2}=0,
  \end{equation*}
   \begin{equation*}
  x^2J_{\nu_1}^{\prime\prime} +xJ_{\nu_1}^{\prime}+(x^2-\nu_1^2)J_{\nu_1}=0.
  \end{equation*}
  Multiplying the first equation by $J_{\nu_1}$  and the second one by $J_{\nu_2}$, we have
  \begin{equation*}
\frac{\nu_2^2}{x} J_{\nu_2}J_{\nu_1}=xJ_{\nu_2}^{\prime\prime}J_{\nu_1}+ J_{\nu_2} ^\prime J_{\nu_1}+x J_{\nu_2}J_{\nu_1},
 \end{equation*}
 \begin{equation*}
  \frac{\nu_1^2}{x}J_{\nu_2}J_{\nu_1}=xJ_{\nu_1}^{\prime\prime}J_{\nu_2}+ J_{\nu_1} ^\prime J_{\nu_2}+x J_{\nu_2}J_{\nu_1}.
  \end{equation*}
    Subtracting the second equality from the first one,
  \begin{equation*}
 [x(J_{\nu_2} ^\prime J_{\nu_1}-J_{\nu_1} ^\prime J_{\nu_2})]^\prime= \frac{(\nu_2^2-\nu_1^2)}{x}J_{\nu_2}J_{\nu_1}.
  \end{equation*}
  Integrating this equation from $0$ to $\tau$, leads  to the assertion.
      \end{proof}

This section is closed with some results from  the rearrangement theory  related to our optimization problems. The reader can refer to  \cite{Alvino,bu89} % bert1,bert2
for further information about the theory of rearrangements.
\begin{defn}\label{readef}
Two Lebesgue measurable functions  $\rho: \Om \rightarrow \Bbb{R}$, $\rho_0:\Om \rightarrow \Bbb{R}$, are said to be rearrangements of each other if\\
 \begin{equation}\label{rea}
|\{x\in \Om : \rho(x)\geq \tau\}|=|\{x\in \Om : \rho_0(x)\geq \tau\}|\qquad~\quad\forall \tau \in \mathbb{R}.
\end{equation}
\end{defn}
The notation $\rho\sim \rho_0$ means that $\rho$ and $\rho_0$ are rearrangements of each other. Consider $\rho_0:\Om \rightarrow \Bbb{R}$, the class of rearrangements generated by $\rho_{0}$, denoted $\mathcal{P}$, is defined as follows\\
 \begin{equation*}
\mathcal{P}=\{\rho:\rho\sim \rho_{0}\}.
\end{equation*}

 Let $\rho_{0}= \beta\chi_{D_0}+\alpha\chi_{D_0^c}$ where $ D_0\subset \Om$ and $|D_0|=A$. For the sake of completeness, we include following
 technical assertion.

\begin{lem}\label{chirho}
  A function $\rho$ belongs to the rearrangement class $\mathcal{P}$ if and only if  $\rho= \beta\chi_{D}+\alpha\chi_{D^c}$ such that $ D\subset \Om$ and $|D|=A$.
 \end{lem}
 \begin{proof}
  Assume $\rho \in \mathcal{P}$. In view of definition \ref{readef},
  \begin{align*}
|\{x\in \Omega: \rho_0(x)= r\}|=|\cap_1 ^\infty \{x\in \Omega: r \leq \rho_0(x)< r +\dfrac{1}{n}\}|\\
=\lim_{n\rightarrow \infty}
|\{x\in \Omega:  \rho_0(x)\geq r\}|- |\{x\in \Omega:  \rho_0(x) \geq r +\dfrac{1}{n}\}|\\
=\lim_{n\rightarrow \infty}
|\{x\in \Omega:  \rho(x)\geq r\}|- |\{x\in \Omega:  \rho(x) \geq r +\dfrac{1}{n}\}|\\
=|\cap_1 ^\infty \{x\in \Omega: r \leq \rho(x)< r +\dfrac{1}{n}\}|= |\{x\in \Omega: \rho(x)= r\}|,
\end{align*}
  where it means that the
  level sets of $\rho$ and $\rho_0$  have the same measures and this yields the assertion. The other part of the theorem is concluded from
  definition \ref{readef}.
 \end{proof}

  Let us state here one of the essential tools in studying rearrangement optimization problems.

\begin{lem}\label{ber}
Let $\mathcal{P}$ be the set of rearrangements of a fixed function
$\rho_{0}\in L^r(\Omega)$, $r>1$, $\rho_{0}\not\equiv 0$, and let $q\in
L^s(\Omega)$, $s=r/(r-1)$, $q\not\equiv 0$.  If there is a decreasing
function $\eta:\mathbb{R}\rightarrow \mathbb{R}$ such that $\eta(q)\in \mathcal{P}$, then
\begin{equation*}
\int_{\Omega} \rho q dx \geq \int_{\Omega} \eta(q)q
dx~~~~~~~\qquad\qquad \forall ~\rho \in\mathcal{P},
\end{equation*}
and the function $\eta(q)$ is the unique minimizer relative to
$\mathcal{P}$.
\end{lem}
\begin{proof}
 See \cite{bu89}.
\end{proof}

     \section{Refusing the conjecture}\label{main result}%%%%%%%%%%%%%%%%%%%%%%%%%%%%%%%%%%%%%%%%%%%%%%%%%%%%%%%%%%%%%%%%%%%%%%%%%%%%%%%%%%%%%%%%%%%%%%%%%%%%

$\quad$ In this section, we investigate the conjecture proposed in
\cite{Conca2} when $\Om$ is a ball in
$\mathbb{R}^n$ such that $n\geq 2$. We show that the conjecture is false not only for $n=2,3$ but also for every $n \geq 4$.
 Indeed, we will establish that a ball could not be a global
minimizer for the optimization problem \eqref{mainopt} when $\al$
and $\bt$ are close to each other (low contrast regime) and $A$ is large enough.  It should be noted
that our method is not as complicated as the approach has been
stated in \cite{Conca3} and we deny the conjecture in a simpler way.

 We hereafter regard $\Om\subset \mathbb{R} ^n$ as the unit ball centered at the origin. Assume that $\psi$
   is the eigenfunction
   corresponding to the principal eigenvalue of the Laplacian
with Dirichlet's boundary condition on
   $\Om $. Then, one can consider $\psi=\psi(r)$ as a radial function which satisfies
      \begin{equation}\label{laplacian}
       \left\{
     \begin{array}{ll}
     r^2\psi^{\prime\prime}(r)+(n-1)r\psi^\prime(r)+\lambda r^2\psi(r)=0\qquad 0<r<1 ,
       \\ \psi^\prime(0)=0 \quad \:\:\: \psi(1)=0,
            \end{array}
   \right.
      \end{equation}
     where the boundary conditions correspond to  the continuity of the gradient at the origin and  Dirichlet's condition
     on the boundary. In the next lemma, we examine the function $|\psi^\prime(r)|$.
     \begin{lem}\label{maxrho}
      Let $\psi$ be the eigenfunction of \eqref{laplacian} associated with the principal eigenvalue $\lambda$. Then, function $|\psi^\prime(r)|$   has a unique maximum point $\rho_n$ in $(0,1)$.
     \end{lem}
     \begin{proof}
      The solution of \eqref{laplacian} is
      \begin{equation*}
      \psi(r)=r^{1-\frac{n}{2}}J_{\frac{n}{2}-1}(\mu r)\qquad 0\leq r \leq 1,
      \end{equation*}
      where $\mu= j_{\frac{n}{2}-1,1}$. For the reader's convenience, we use the change of variable $t=\mu r$ and then

      \begin{equation*}
       \psi(t)=\mu^{\frac{n}{2}-1}\left(\frac{J_{\frac{n}{2}-1}(t)}{t^{\frac{n}{2}-1}}\right)\qquad 0\leq t \leq \mu.
          \end{equation*}
            According to lemma \ref{besselroot}, $j_{\frac{n}{2}-1,1}<j_{\frac{n}{2},1}$ and then we see $J_\frac{n}{2}(t)\geq 0$ for $0\leq t \leq \mu$. Therefore,
      \begin{equation*}
      |\psi^\prime(t)|=\mu^{\frac{n}{2}-1}\left(\frac{J_{\frac{n}{2}}(t)}{t^{\frac{n}{2}-1}}\right)\qquad 0\leq t \leq \mu,
      \end{equation*}
       invoking formula \eqref{besselrec2}.  To determine the maximum point of this function, one should calculate $\frac{d}{dt} (|\psi^\prime(t)|)$. Employing relations \eqref{besselrec1} and  \eqref{besselrec2},
       \begin{equation*}
      \frac{d}{dt} (|\psi^\prime(t)|)=\frac{\mu^{\frac{n}{2}-1}(t J_{\frac{n}{2}-1}(t)-(n-1)J_{\frac{n}{2}}(t))}{t^{\frac{n}{2}}}.
      \end{equation*}
      Then $\frac{d}{dt} (|\psi^\prime(t)|)=0$ yields
       \begin{equation*}
      t J_{\frac{n}{2}-1}(t)-(n-1)J_{\frac{n}{2}}(t)=0.
      \end{equation*}
      The zeros of the last equation are the fixed points of the function
       \begin{equation*}
     g(t)=(n-1) \frac{J_{\frac{n}{2}}(t)}{J_{\frac{n}{2}-1}(t)}   \qquad 0< t < \mu.
      \end{equation*}
      We find that
      \begin{equation*}
       J^\prime_{\frac{n}{2}}(t)J_{\frac{n}{2}-1}(t)-J_{\frac{n}{2}}(t)J^\prime_{\frac{n}{2}-1}(t) =\frac{(n-1)}{t} \int_0^t \frac{J_{\frac{n}{2}}(\tau)J_{\frac{n}{2}-1}(\tau)}{\tau}d\tau,
      \end{equation*}
      applying lemma \ref{goodrel}. Consequently, $g^\prime(t)>0$ for $0<t<\mu$ and $g$ is an increasing function. On the other hand,
      $g(t)$ tends to infinity when $t\rightarrow \mu$ and, in view of formula \eqref{besselseri},  it tends to zero when $t\rightarrow 0$.
      Thus, $g(t)$ has a unique fixed point $\rho_n$ in $(0,\mu)$ which it is the unique extremum point of $|\psi^\prime(t)|$. Recall
      that  $t J_{\frac{n}{2}-1}(t)-(n-1)J_{\frac{n}{2}}(t)$ is negative when $t\rightarrow \mu$. Hence, $\frac{d}{dt} (|\psi^\prime(t)|)$ is negative
      in a neighborhood of $\mu$ and thus, $\rho_n$ is the unique maximum point of $\frac{d}{dt} (|\psi^\prime(t)|)$ in $(0,\mu)$.
     \end{proof}

We need the following theorem to deduce the main result.
\begin{thm}\label{lurianthem}
Assume  $D_0$ is a  subset of $\Om$ where $|D_0|=A$ and $u_0$ is the eigenfunction of \eqref{mainpde} corresponding to $\la(D_0)$. Let $D_1$
be a subset of $\Om$ where
\begin{equation}\label{tformulap}
|D_1|=A \,\:\mathrm{and}\,\:D_1=\{x:\,\:|\nabla u_0|\leq t\}
\end{equation}
with
\begin{equation}\label{tformula}
t=\inf\{s\in \mathbb{R} : |\{x:\,\:|\nabla u_0|\leq s\}|\geq A\}.
\end{equation}
Then, $\la(D_1)\leq \la(D_0)$.
\end{thm}
\begin{proof}
 It is well known, from the Krein-Rutman theorem \cite{rutman}, that $u_0$ is positive everywhere on $\Om$. Therefore,
 we infer that  all sets  $\{x:\,\:|\nabla u_0|=s\}$
 have  measure zero because of  lemma 7.7 in
\cite{gilbarg}. Then, one can determine set $D_1$ uniquely using the above formula. Let us define the  following
decreasing function
\begin{equation*}
 \eta(s)=
 \left\{
     \begin{array}{ll}
      \beta\quad\;\;\;\:\quad 0 \leq s\leq t^2,
              \\ \alpha \quad  \quad \quad\quad \quad s>t^2,
     \end{array}
   \right.
 \end{equation*}
 where  it yields
 \begin{equation*}
 \eta(|\nabla u_0|^2)=\beta \chi_{D_1}+\alpha \chi_{D_1^c}.
 \end{equation*}
 Employing lemma  \ref{chirho} and \ref{ber}, we can deduce
 \begin{equation*}
 \int (\beta\chi_{D_1}+\alpha \chi_{D_1^c}) |\nabla u_0|^2dx\leq\int (\beta\chi_{D_0}+\alpha \chi_{D_0^c}) |\nabla u_0|^2dx,
 \end{equation*}
 and then we have $\la(D_1)\leq \la(D_0)$ invoking \eqref{lambdavar}.
\end{proof}

\begin{rem}\label{levelsetofu}
In theorem \ref{lurianthem},  if $D_1\neq D_0$, then
\begin{equation*}
 \int (\beta\chi_{D_1}+\alpha \chi_{D_1^c}) |\nabla u_0|^2dx<\int (\beta\chi_{D_0}+\alpha \chi_{D_0^c}) |\nabla u_0|^2dx,
 \end{equation*}
applying the uniqueness of the minimizer in lemma \ref{ber}. Thus, we observe that $\la(D_1)< \la(D_0)$ when $D_1\neq D_0$.

\end{rem}
\begin{rem}\label{approximant}
In \cite{Conca3}, it has been proved that if ${\rho}_*=\beta\chi_{D_*}+\alpha \chi_{D_*^c}$ is the minimizer of
\begin{equation}\label{appproblem}
\underset{\rho \in\mathcal{P}}{\min} \int_{\Omega} \rho |\nabla \psi|^2 dx,
\end{equation}
then the set $D_*$ is an approximate solution for \eqref{mainopt}, under the assumption of low contrast regime. By arguments similar to
those in the proof of theorem \ref{lurianthem}, one can determine the unique minimizer of problem \eqref{appproblem}, ${\rho}_*=\beta\chi_{D_*}+\alpha \chi_{D_*^c}$,
using formulas \eqref{tformulap} and \eqref{tformula}. Recall from lemma \ref{maxrho} that $|\psi^\prime(r)|$ has a unique maximum point $\rho_n$ in $(0,1)$
and it is a continuous function on $[0,1]$ with $|\psi^\prime(0)|=0$. Then the unique symmetrical domain $D_*$ which ${\rho}_*=\beta\chi_{D_*}+\alpha \chi_{D_*^c}$ is the solution of \eqref{appproblem} is of two possible types. The set $D_*$ is a ball centered at the origin if $A\leq |\mathcal{B}(0, \rho_n)|$
and it is the union of a ball and an annulus touching the outer boundary of $\Om$ if $A> |\mathcal{B}(0, \rho_n)|$.

 This result has been established in \cite{Conca3} for $n=2,3$.
\end{rem}
 Now we are ready to state the main result. Indeed, we establish that locating the material with the highest
 conductivity in a ball centered at the origin is not the minimal distribution since we can find another radially symmetric
distribution of the materials which has a smaller basic frequency.

\begin{thm}\label{bestresult}
Let
$D_0=\mathcal{B}(0,\rho)\subset \Om$ be a ball centered at the origin with
$|D_0|=A$. If  $\beta$ is sufficiently close to $\alpha$ and $\rho> \rho_n$, then there is a set $D_1\subset \Om$ with
$|D_1|=A$  containing a radially symmetric subset of
$D_0^c$ where $\la(D_1)<\la(D_0)$.
\end{thm}
\begin{proof}
Suppose $u_0$ is the eigenfunction of \eqref{mainpde} associated with  $\la=\la(D_0)$ such that $\|u_0\|_{\lt}=1$. Utilizing theorem \ref{lurianthem} and remark \ref{levelsetofu}, we conclude
$\la(D_1)<\la(D_0)$ provided
\begin{equation*}
D_1=\{x:\,\,|\nabla u_0|\leq
 t\}, \quad t=\inf\{s\in \mathbb{R} : |\{x:\,\:|\nabla u_0|\leq s\}|\geq
 A\},
\end{equation*}
and $D_0\neq D_1$. One can
 observe that $u_0$  satisfies the following transmission problem
\begin{align} \label{traneq}
 \left \{
     \begin{array}{ll}
       -\bt \Delta v_1 = \la v_1\quad\quad \mathrm{in}\quad D_0
       \\ -\al \Delta v_2 = \la v_2\quad\quad \mathrm{in}\quad       D_0^c
       \\v_1(x)=v_2(x)\;\;\;\, \quad\mathrm{on}\;\;\; \,\partial D_0
       \\ \bt\frac{\partial}{\partial \mathfrak{n}} v_1=\al\frac{\partial}{\partial \mathfrak{n}}
       v_2\quad\: \mathrm{on}\quad \:\partial D_0
       \\v_2(x)=0\quad\:\:\:\qquad \mathrm{on}\quad       \partial \Om,
     \end{array}
   \right.
\end{align}
where $\mathfrak{n}$ is the unit outward normal. According to the
above representation, $u_0$ is an analytic function in the closure
of sets $D_0$ and $D_0^c$ employing the analyticity theorem
\cite{john}.

  We should  assert
that $D_0\neq D_1$. To this end, let us note that $u_0$ is a radial
function and so $u_0(x)=y(r)$, $r=\|x\|,$ where the function $y$
solves
\begin{equation} \label{radialeq}
 \left \{
     \begin{array}{ll}
       y^{\prime\prime}(r)+\frac{n-1}{r}y^{\prime}(r)+\frac{\la}{\bt} y(r)
       =0 \quad \mathrm{in}\;\;(0,\rho)
       \\y^{\prime\prime}(r)+\frac{n-1}{r}y^{\prime}(r)+\frac{\la}{\al} y(r)
       =0 \quad \mathrm{in}\;\;(\rho,1)
       \\y(\rho^-)=y(\rho^+)
       \\ \bt y^{\prime}(\rho^-)=\al
       y^{\prime}(\rho^+)
       \\y^{\prime}(0)=0,\:\:y(1)=0.
            \end{array}
   \right.
\end{equation}

We   introduce $y_1(r)$ and $y_2(r)$ as the solution of
\eqref{radialeq} in $[0, \rho]$ and $[\rho, 1]$ respectively. We claim
that if

\begin{equation}\label{maininequality}
|y_2^{\prime}(1)|< z=\underset{r\in[0,\rho]}{\max} |y_1^{\prime}(r)|,
\end{equation}
then $D_1$ contains a radially symmetric subset of
$D_0^c$ and so $D_1$ is not equal to $D_0$.

  Recall that
level sets of $|\nabla u_0|$ have  measure zero.  Hence,  if $|y_2^{\prime}(r)|>z$ for all $r$ in $[\rho,1]$
then $D_1=\{x:\; |\nabla u_0|\leq t\}=D_0$ with $t=z$.  On the
other hand, if $|y_2^{\prime}(1)|<z$ then we have $t<z$ to satisfy
the condition $|D_1|=A$, in view of the continuity of the function $|y_2^{\prime}(r)|$  . In other words, $D_1$
 should include a radially symmetric subset of
$D_0^c$. This discussion proves our claim.

 It remains to verify inequality \eqref{maininequality}. This is a standard result of the perturbation theory
 of eigenvalues that $u_0$ tends to $\psi$ with $\|\psi\|_{\lt}=1$ and $\lambda$ converges to $\al \mu$ when $\beta$
 decreases to $\alpha$ \cite{rellich}. The convergence of the eigenfunctions holds in the space $H^1_0(\Om)$. Hence it yields that
 $y(r)$ and $y^\prime(r)$ converge to $\psi(r)$ and $\psi^\prime(r)$ almost everywhere in $\Om$, respectively. Since
 $y^\prime(r)$ and $\psi^\prime(r)$ are continuous functions on the sets $[0,\rho]$ and $[\rho,1]$, the convergence
 is pointwise\cite{lieb}. In summary, $|y_1^\prime(r)|$ converges to $|\psi^\prime(r)|$ pointwise for all $r$ in $[0, \rho]$
 and $|y_2^\prime(r)|$ converges to $|\psi^\prime(r)|$ pointwise in $[\rho, 1]$. Additionally, $\left||y_2^\prime(\rho)|-|y_1^\prime(\rho)|\right|$
  converges to zero when $\beta$ approaches $\alpha$. Invoking lemma \ref{maxrho}, we see that $|\psi^\prime(\rho)|-|\psi^\prime(1)|=d_n>0$
 when $\rho>\rho_n$. Thus, if $\beta$ is close to $\alpha$ enough, we have
 \begin{equation}\label{lastrel1}
 \left||y_2^\prime(\rho)|-|y_2^\prime(1)|\right|> {d_n}/2,
 \end{equation}
 and also
 \begin{equation}\label{lastrel2}
 |y_2^\prime(\rho)|\rightarrow |\psi^\prime(\rho)|,\quad  |y_2^\prime(1)|\rightarrow |\psi^\prime(1)|,\quad |y_2^\prime(\rho)|\rightarrow |y_1^\prime(\rho)|,
 \end{equation}
  as $\beta$ converges to $\alpha$. Applying \eqref{lastrel1} and \eqref{lastrel2}, leads us to  inequalities
  \begin{equation*}
  |y_2^\prime(1)|<|y_1^\prime(\rho)| \leq z.
  \end{equation*}

 \end{proof}

%%%%%%%%%%%%%%%%%%%%%%%%%%%%%%%%%%%%%%%%%%%%%%%%%%%%%%%%%%%%%%%%%%%%%%%%%%%%%%%%%%%%%%%%%%%%%%%%%%%%%%%%%%%%%%%%%%%%

%%%%%%%%%%%%%%%%%%%%%%%%%%%%%%%%%%%%%%%%%%%%%%%%%%%%%%%%%%%%%%%%%%%%%%

%\section{References}

\end{document}